\documentclass[a4paper,11pt]{amsart}
\usepackage{amsfonts,amssymb,amsthm,cite,amsmath,amstext}
\usepackage{color}
\usepackage[colorlinks,linkcolor=black,citecolor=black]{hyperref}
\usepackage{comment}
\usepackage{framed}

\definecolor{shadecolor}{gray}{0.875}

\newtheorem{thrm}{Theorem}[section]
\newtheorem{thrmx}{Theorem}

\newtheorem{corx}{Corollary}

\newtheorem{lem}[thrm]{Lemma}

\newtheorem{prop}[thrm]{Proposition}

\theoremstyle{definition}
\newtheorem{defn}[thrm]{Definition}

\newtheorem{rmk}[thrm]{Remark}

\title{Mixed Hodge-Riemann bilinear relations and m-positivity}
\author{Jian Xiao}
\date{}

\begin{document}
\maketitle

\begin{abstract}
Motivated by our previous work on Hodge-index type inequalities, we give a form of mixed Hodge-Riemann bilinear relation by using the notion of $m$-positivity, whose proof is an adaptation of the works of Timorin and Dinh-Nguy\^{e}n. This mixed Hodge-Riemann bilinear relation holds with respect to mixed polarizations in which some satisfy particular positivity condition, but could be degenerate along some directions. In particular, it applies to fibrations of compact K\"ahler manifolds.
\end{abstract}

\tableofcontents

\section{Introduction}
\subsection{The classical and mixed HRR}
Let $X$ be a compact K\"ahler manifold of dimension $n$, and let $\omega$ be a K\"ahler class on $X$. Let $0\leq p,q \leq p+q \leq n$ be integers. Denote $\Omega= \omega^{n-p-q}$, which is a strictly positive class in $H^{n-p-q, n-p-q}(X, \mathbb{C})$. Associated to $\Omega$, one could define the following quadratic form $Q$ on $H^{p,q}(X, \mathbb{C})$:
\begin{equation*}
  Q(\Phi, \Psi) = i^{q-p} (-1)^{(p+q)(p+q+1)/2} \Omega \cdot \Phi \cdot \overline{\Psi}.
\end{equation*}
The classical Hodge-Riemann bilinear relation theorem (HRR) (see e.g. \cite{Dem_AGbook}, \cite{voisinHodge1}) states that $Q$ is positive definite on the primitive subspace $P^{p,q}(X, \mathbb{C})$, which is defined as follows:
\begin{equation*}
  P^{p,q}(X, \mathbb{C}) = \{\Phi \in H^{p,q} (X, \mathbb{C})| \Omega \cdot \omega \cdot \Phi =\omega^{n-p-q+1}\cdot \Phi=0\}.
\end{equation*}
As a corollary, one could get the classical Hard Lefchetz theorem (HL), i.e., the linear map
\begin{align*}
  \Omega: &H^{p,q}(X, \mathbb{C})\rightarrow H^{n-q,n-p}(X, \mathbb{C}),\\
   &\Phi \mapsto \Omega \cdot \Phi
\end{align*}
is an isomorphism.
Furthermore, by the classical HRR, one could also get the Lefchetz decomposition theorem (LD), that is, there is an orthogonal decomposition
\begin{equation*}
  H^{p, q}(X, \mathbb{C}) = P^{p,q}(X, \mathbb{C}) \oplus \omega \wedge H^{p-1, q-1}(X, \mathbb{C})
\end{equation*}
with respect to $Q$, with the convention that $H^{p-1, q-1}(X, \mathbb{C})=\{0\}$ if either $p=0$ or $q=0$.

In \cite{DN06,cattanimixedHRR} (see also \cite{gromov1990convex,timorinMixedHRR}), the classical HRR is greatly generalized to the mixed situation. More precisely, let $\omega_1, ...,\omega_{n-p-q+1}$ be K\"ahler classes on $X$. Denote $\Omega=\omega_1 \cdot ...\cdot \omega_{n-p-q}$. We call $\Phi \in H^{p,q} (X, \mathbb{C})$ primitive with respect to $\Omega \cdot \omega_{n-p-q+1}$, if $$\Omega\cdot \omega_{n-p-q+1} \cdot \Phi =0.$$ The mixed HRR states that the corresponding quadratic form $Q$ defined by $\Omega$ is positive definite on the subspace of primitive classes. This mixed HRR then implies the mixed versions of HL and LD.

The reader can find some applications of this mixed HRR in complex geometry in \cite{DinhS04aut, dinhS05greencurrent}. One can also find some related topics and applications in the other contexts in \cite{williamsonHodgeSeorgel}, \cite{williamson2016hodge}, \cite{huhHRR}, \cite{mcmullenSimplePolytopes}, \cite{timorinPolytopeHRR}, \cite{cataldoMigHLsemismallmap} and the references therein.

\subsection{The main result}
Our aim is to weaken the positivity of $(1,1)$ classes, which we also call polarizations, in the definition of $\Omega$ such that the HRR still holds. The key positivity notion is $m$-positivity.

Let $\omega$ be a reference K\"ahler metric on $X$ of dimension $n$. Then a real smooth $(1,1)$ form $\widehat{\alpha}$ is called $m$-positive $(1\leq m \leq n)$ with respect to $\omega$, if
$$\widehat{\alpha}^k \wedge \omega^{n-k}>0$$
holds for any $1\leq k \leq m$ and any point on $X$. It is well-known that $\widehat{\alpha}$ is $n$-positive if and only if it is strictly positive.
This kind of positivity plays an important role in the study of Hessian equations, as it gives the natural solution set of such PDEs.
A cohomology class $\alpha\in H^{1,1}(X, \mathbb{R})$ is called $m$-positive with respect to the metric $\omega$, if it has a smooth representative which is $m$-positive in the pointwise sense.

In our previous work \cite{xiao18Hodgeindex}, inspired by the Hodge-index type inequalities obtained by complex Hessian equations, we proved a Hodge-index type theorem for classes of type $(1,1)$ by using $m$-positivity and Garding's theory of hyperbolic polynomials. More precisely, let $\alpha_1,...,\alpha_{m-1}$ be $m$-positive classes, and denote $$\Omega = \omega^{n-m}\cdot \alpha_1 \cdot...\cdot \alpha_{m-2},$$
then the HRR holds with respect to $\Omega$, where the primitive subspace is defined by $\Omega\cdot \alpha_{m-1}$. This is the main motivation of our work.

We generalize this result to arbitrary $(p,q)$ classes, by assuming further that the $\alpha_j$ are semipositive. Note that: if an $(1,1)$ form is semipositive, then it is $m$-positive with respect to $\omega$ if and only if it has at least $m$ positive eigenvalues.
Thus our positivity assumption is that every $\alpha_j$ has a semipositive smooth representative, which has at least $m$ positive eigenvalues at every point of $X$.

\begin{thrmx}\label{thrm main result}
Let $X$ be a compact K\"ahler manifold of dimension $n$, and let $\omega$ be a K\"ahler class on $X$. Let $p, q, m$ be integers such that $0\leq p,q \leq p+q \leq m \leq n$.
Let $\alpha_1, ...,\alpha_{m-p-q+1} \in H^{1,1}(X, \mathbb{R})$. Assume that every $\alpha_j$ has a smooth representative which is semipositive and has at least $m$ positive eigenvalues at every point. Denote $\Omega=\omega^{n-m}\cdot \alpha_1 \cdot \alpha_2 \cdot ...\cdot \alpha_{m-p-q}$.
Let
\begin{equation*}
  P^{p,q}(X, \mathbb{C})=\{\Phi \in H^{p,q} (X, \mathbb{C})| \Omega \cdot \alpha_{m-p-q+1} \cdot \Phi =0\}
\end{equation*}
be the primitive subspace defined by $\Omega \cdot \alpha_{m-p-q+1}$.
Then
\begin{itemize}
\item (HRR) the quadratic form
\begin{equation*}
  Q(\Phi, \Psi) = i^{q-p} (-1)^{(p+q)(p+q+1)/2} \Omega \cdot \Phi \cdot \overline{\Psi}
\end{equation*}
is positive definite on $P^{p,q}(X, \mathbb{C})$.
\end{itemize}
This implies that
\begin{itemize}
   \item (HL) the map
   \begin{equation*}
     \Omega: H^{p,q} (X, \mathbb{C}) \rightarrow H^{n-q,n-p} (X, \mathbb{C}),\ \Phi \mapsto \Omega\cdot \Phi
   \end{equation*}
   is an isomorphism.
  \item (LD) the space $H^{p,q} (X, \mathbb{C})$ has an orthogonal decomposition
  \begin{equation*}
    H^{p,q} (X, \mathbb{C}) = P^{p,q}(X, \mathbb{C}) \oplus \alpha_{m-p-q+1} \wedge H^{p-1, q-1}(X, \mathbb{C})
  \end{equation*}
  with respect to $Q$, and $\dim P^{p,q} (X, \mathbb{C}) = \dim H^{p,q} (X, \mathbb{C}) - \dim H^{p-1,q-1} (X, \mathbb{C})$, where we use the convention that $H^{p-1, q-1}(X, \mathbb{C}) =\{0\}$ if either $p=0$ or $q=0$.
\end{itemize}

\end{thrmx}

\begin{rmk}
In general, the HRR, HL and LD theorems are not true if $\Omega$ is an arbitrary class in $H^{n-p-q,n-p-q} (X, \mathbb{C})$, even if the class has a smooth strictly positive representative (see e.g. \cite[Section 9]{berdtsibony02dbarcurrent}, \cite[Remark 2.9]{DN13HRR}). Thus Theorem \ref{thrm main result} provides a sufficient condition on $\Omega$ such that HRR, HL and LD hold true.
\end{rmk}

\subsection{Relative HRR}
Interesting examples are given by the holomorphic fibrations between compact K\"ahler manifolds.

\begin{corx}\label{cor relative HRR}
Let $f: X\rightarrow Y$ be a holomorphic submersion from a compact K\"ahler manifold of dimension $n$ to a compact K\"ahler manifold of dimension $m$. 
Let $p, q, m$ be non-negative integers such that $p+q \leq m$. Assume that $\omega_X$ is a K\"ahler class on $X$ and $\omega_{Y_1},...,\omega_{Y_{m-p-q+1}}$ are K\"ahler classes on $Y$, then the HRR holds with respect to $$\Omega_{X, Y}=\omega_X ^{n-m} \cdot f^*\omega_{Y_1}\cdot... \cdot f^* \omega_{Y_{m-p-q}},$$
where the primitive space is defined by $\Omega_{X, Y} \cdot f^* \omega_{Y_{m-p-q+1}}$.
\end{corx}

This is true because every $f^*\omega_{Y_k}$ is $m$-positive with respect to some K\"ahler metric on $X$ and semipositive. In some sense, this can be seen as a relative version of the classical and mixed HRR.

\begin{rmk}
When $f$ is just a surjective holomorphic map, the $f^*\omega_{Y_j}$ is $m$-positive at generic points, that is, it has a smooth representative which is $m$-positive on a Zariski open set. In this general setting, we are not quite sure what kind of condition could be proposed to $f$ such that the ``relative'' HRR holds if and only if the condition on $f$ holds. Note that the relative HRR or HL is not true for a general map. Nevertheless, for general holomorphic maps, see Section \ref{sec semismall} for some discussions. We intend to address this in a future work.

To our knowledge, there are only some known results (see \cite{cataldoMigHLsemismallmap}) when $f:X\rightarrow Y$ is a semi-small map between projective varieties and the $\omega_{Y_j}$ are given by the same ample line bundle on $Y$. For this particular case, $f$ is generically finite, thus $n=m$.
\end{rmk}

\subsection{About the proof}
To prove the main result, we mainly follow the approach of Timorin \cite{timorinMixedHRR} and Dinh-Nguy\^{e}n \cite{DN06, DN13HRR}. First, we apply Timorin's inductive argument to establish the linear version of our HRR. To this end, we need to study the positivity of restrictions of $m$-positive forms. This is the only place where we need to assume further that the $\alpha_j$ are semipositive, not only $m$-positive. Once the linear HRR is proved, then we could apply the $dd^c$-method to reduce the global case to the local case.

\section{Preliminaries}

\subsection{Restriction of $m$-positive forms}
In this section, we consider the problem on positivity of the restrictions of $m$-positive forms. Let $\Lambda^{1,1}_{\mathbb{R}} (\mathbb{C}^n)$ be the space of real $(1,1)$ forms on $\mathbb{C}^n$ with constant coefficients.

\begin{lem}\label{lem restc}
Let $\omega$ be a K\"ahler metric on $\mathbb{C}^n$ with constant coefficients. Assume that $\alpha \in \Lambda^{1,1}_{\mathbb{R}} (\mathbb{C}^n)$ is $m$-positive with respect to $\omega$. Then
\begin{itemize}
  \item for any $1\leq k \leq m-1$ and any hyperplane $H\subset \mathbb{C}^n$, we have $\alpha_{|H} ^k \wedge \omega_{|H} ^{n-k-1}>0$ on $H$.
  \item for $k=m$, if we assume further that $\alpha$ is semipositive, then there is a proper subspace $S(\alpha)$ such that for any $v\in \mathbb{C}^n \setminus S(\alpha)$, we have $\alpha_{|H_v} ^m \wedge \omega_{|H_v} ^{n-m-1}>0$ on $H_v$, where $$H_v = \{z\in \mathbb{C}^n| v\cdot z =0\}$$ is the hyperplane defined by $v$.
\end{itemize}

\end{lem}

\begin{proof}
Let $a = (a_1,...,a_n)$ be a non-zero vector in $\mathbb{C}^n$. Let $H_a$ be the hyperplane defined by $a$. Denote the equation of $H_a$ by the same notation $H_a (z) = a \cdot z$.

Since $a\neq 0$, we could find linear functions $w_1,...,w_{n-1}$ of $z$ such that $(w_1,...,w_{n-1}, H_a)$ give a new coordinate system of $\mathbb{C}^n$. Write $\alpha = \alpha_0 + \beta_0$, where $\beta_0$ is the sum of $(1,1)$ forms containing either $dH_a$ or $d\overline{H_a}$. Then $\alpha_0 = \alpha_{|H_a}$. Similarly, we write $\omega = \omega_{|H_a} + \beta_1$. It is easy to see that
\begin{equation}\label{eq restc}
  \alpha^k \wedge \omega^{n-k-1} \wedge i dH_a \wedge d\overline{H_a} = \alpha_{|H_a} ^k \wedge \omega_{|H_a} ^{n-k-1} \wedge i dH_a \wedge d\overline{H_a}.
\end{equation}
In particular, if we denote $V=i^{n-1} dw_1 \wedge d\bar w_1 \wedge...\wedge dw_{n-1} \wedge d\bar w_{n-1}$, then (\ref{eq restc}) yields
\begin{equation}\label{eq restc1}
 \frac{ \alpha^k \wedge \omega^{n-k-1} \wedge i dH_a \wedge d\overline{H_a}}{ V \wedge i dH_a \wedge d\overline{H_a}} = \frac{\alpha_{|H_a} ^k \wedge \omega_{|H_a} ^{n-k-1}}{V}.
\end{equation}
For simplicity, we just write $(\ref{eq restc1})$ as
\begin{equation}\label{eq restc2}
  \alpha^k \wedge \omega^{n-k-1} \wedge i dH_a \wedge d\overline{H_a} = \alpha_{|H_a} ^k \wedge \omega_{|H_a} ^{n-k-1}.
\end{equation}

We first prove the first statement. For $1\leq k \leq m-1$, we can write
\begin{equation*}
  \alpha^k \wedge \omega^{n-k-1} \wedge i dH_a \wedge d\overline{H_a} =\omega^{n-m}\wedge \omega^{m-k-1}  \wedge \alpha^k \wedge  \wedge i dH_a \wedge d\overline{H_a}.
\end{equation*}
Since $\omega, \alpha$ are $m$-positive with respect to $\omega$, by the theory of hyperbolic polynomials (see e.g. \cite{gardinghyperbolic}, \cite[Chapter 2]{hormanderConvexity}, \cite[Lemma 3.8]{xiao18Hodgeindex}), $\alpha^k \wedge \omega^{n-k-1}$ is strictly positive. Thus for any non-zero semipositive $(1,1)$ form $\beta$,
\begin{equation*}
  \alpha^k \wedge \omega^{n-k-1} \wedge \beta>0.
\end{equation*}
Letting $\beta = i dH_a \wedge d\overline{H_a}$ and using (\ref{eq restc2}) finish the proof.

Next we consider the second statement.
Assume that
\begin{equation*}
 \alpha^m \wedge \omega^{n-m-1} = \sum_{i,j} \Phi_{ij} \widehat{dz_i \wedge d\bar z_j},
\end{equation*}
where $\widehat{dz_i \wedge d\bar z_j}$ is the $(n-1, n-1)$ form omitting $dz_i, d\bar z_j$ such that $$\widehat{dz_i \wedge d\bar z_j} \wedge i dz_i \wedge d\bar z_j >0.$$
After dividing a volume form, we have
\begin{equation*}
 \alpha^m \wedge \omega^{n-m-1} \wedge i dH_a \wedge d\overline{H_a} = \sum_{ij} \Phi_{ij} a_i \overline{a_j}.
\end{equation*}

As we have assumed further that $\alpha$ is semipositive, the Hermitian matrix $[\Phi_{ij}]$ is semipositive and non-zero. This implies that the set
\begin{equation*}
  S(\alpha) = \{(a_1,...,a_n)\in \mathbb{C}^n | \sum_{ij} \Phi_{ij} a_i \overline{a_j} =0\}
\end{equation*}
is a proper linear subspace of $\mathbb{C}^n$. Thus for any $v\in \mathbb{C}^n \setminus S(\alpha)$, we have $$\alpha^m \wedge \omega^{n-m-1} \wedge i dH_v \wedge d\overline{H_v}>0.$$
By (\ref{eq restc2}), this is equivalent to
\begin{equation*}
  \alpha_{|H_v} ^m \wedge \omega_{|H_v} ^{n-m-1}>0.
\end{equation*}

This finishes the proof of the second statement.
\end{proof}

\begin{rmk}
Take a coordinate system such that $\omega = i \sum_{j=1} ^n dz_j \wedge d\bar z_j$, $\alpha=i \sum_{j=1} ^n \lambda_j dz_j \wedge d\bar z_j$. If $\alpha$ is only $m$-positive with respect to $\omega$, then (\ref{eq restc2}) implies that $\alpha_{|H_e}$ is $m$-positive with respect to $\omega_{|H_e}$ for any
$$e\in \{(a_1,...,a_n)| a_i = \pm 1\}.$$
To verify this, we only need to observe that $$\alpha^k \wedge \omega^{n-k-1} \wedge i dH_e \wedge d\overline{H_e} = \alpha^k \wedge \omega^{n-k}.$$
In particular, this shows that there is an open set $\mathcal{O}_1$ (independent of $\alpha$) of hyperplanes containing all the $H_e$ such that, for any $H\in \mathcal{O}_1$, $\alpha_{|H}$ is $m$-positive with respect to $\omega_{|H}$. By induction, this yields that there is an open set $\mathcal{O}_m$ (independent of $\alpha$) of linear subspaces of dimension $m$ such that, for any $V\in \mathcal{O}_m$, $\alpha_{|V}$ is K\"ahler on $V$.
\end{rmk}

\subsection{Existence of orthogonal bases}
In order to prove the linear HRR by induction, we need the following existence result.

\begin{lem}\label{lem basis}
Let $H_{v_1},..., H_{v_k} \subset \mathbb{C}^n$ be hyperplanes, then there is an orthonormal basis $(e_1,...,e_n)$ such that every $e_i \in \mathbb{C}^n \setminus \cup_{i=1} ^k H_{v_i}$.
\end{lem}

\begin{proof}
We first take $$e_1 \in \mathbb{C}^n \setminus (\cup_{i=1} ^k H_{v_i} \bigcup \cup_{i=1} ^k \mathbb{C}v_i).$$ Then for every $i$, $H_{v_i} \cap H_{e_1}$ is a hyperplane in $H_{e_1}$. By induction, there is an orthonormal basis $(e_2,...,e_n)$ such that every $$e_i \in H_{e_1}\setminus \bigcup_{i=1} ^k (H_{v_i} \cap H_{e_1}).$$
The vectors $e_1,...,e_n$ give the desired basis.
\end{proof}

\section{Proof of the main result}

\subsection{The local case}
We first adapt the arguments of Timorin \cite{timorinMixedHRR} to give a form of HRR in the linear case. Roughly speaking, Timorin's proof goes by induction:
\begin{itemize}
  \item assume that HRR holds for $\dim \leq n-1$, then it can be used to prove HL for $\dim =n$;
  \item the HL for $\dim =n$ yields HRR for $\dim =n$.
\end{itemize}

\begin{thrm}\label{thrm linear HRR}
Let $\omega\in \Lambda^{1,1} _\mathbb{R} (\mathbb{C}^n)$ be a K\"ahler metric on $\mathbb{C}^n$. Let $p, q$ be arbitrary integers satisfying $0\leq p,q\leq p+q \leq m \leq n$. Assume that $\alpha_1,...,\alpha_{m-p-q+1} \in \Lambda^{1,1} _\mathbb{R} (\mathbb{C}^n)$ are $m$-positive with respect to $\omega$ and semipositive. Denote $$\Omega = \omega^{n-m}\wedge \alpha_1 \wedge...\wedge\alpha_{m-p-q}.$$
Then the HRR holds with respect to $\Omega$, where the primitive space is defined by $\Omega \wedge \alpha_{m-p-q+1}$.
\end{thrm}

In the linear case, the quadratic form $Q$ is defined by
$Q(\Phi, \Psi) = \Omega \wedge \Phi \wedge \overline{\Psi}/\Gamma$,
where $\Gamma$ is a fixed volume form of $\mathbb{C}^n$.
Note that when $n=m$, Theorem \ref{thrm linear HRR} is exactly the linear version of mixed HRR recalled in the introduction.

\begin{rmk}
It is clear that the linear version of HRR is equivalent to the HRR on a compact complex torus.
\end{rmk}

We denote the space of complex $(p,q)$ forms on $\mathbb{C}^n$ with constant coefficients by $\Lambda^{p,q}$.

As stated above, Theorem \ref{thrm linear HRR} will be proved by induction on dimensions.

\begin{lem}\label{lem HL}
Assume Theorem \ref{thrm linear HRR} holds for $\dim = n-1$.
Then the HL holds for $\dim =n$, i.e., $\Omega: \Lambda^{p,q} \rightarrow \Lambda^{n-q, n-p}$ is an isomorphism.
\end{lem}

\begin{proof}
When $p+q =m$ or $m=n$, it is the classical HL recalled in the introduction. Thus we only need to consider the case when $p+q <m<n$.

Without loss of generalities, we can assume $\omega = i \sum_{j=1} ^n dz_j \wedge d\bar z_j$.

We only need to prove that the map defined by $\Omega$ is injective.
Assume that $\Phi \in \Lambda^{p,q}$ satisfies $\Omega \wedge \Phi =0$, then for any hyperplane $H$,
\begin{equation*}
  \omega_{|H}^{n-m-1}\wedge \omega_{|H}\wedge {\alpha_1}_{|H} \wedge...\wedge{\alpha_{m-p-q}}_{|H} \wedge \Phi_{|H}=0.
\end{equation*}
This implies that $\Phi_{|H}$ is primitive on $H$, where $\alpha_{m-p-q+1} = \omega_{|H}$ on $H$.

By Lemma \ref{lem restc}, for every general hyperplane $H=H_v$, where
$$v \in \mathbb{C}^n \setminus \bigcup_{j=1}^{m-p-q} S(\alpha_j),$$
the restrictions ${\alpha_j}_{|H}$ are $m$-positive with respect to $\omega_{|H}$. The restrictions ${\alpha_j}_{|H}$ are also semipositive.

Denote $c=i^{q-p} (-1)^{(p+q)(p+q+1)/2}$. By induction, the HRR holds in lower dimensions, thus
\begin{align}\label{eq induct0}
  Q_H (\Phi_{|H}, \Phi_{|H}) &=  c \omega_{|H}^{n-m-1}\wedge {\alpha_1}_{|H} \wedge...\wedge{\alpha_{m-p-q}}_{|H} \wedge \Phi_{|H} \wedge \overline{\Phi_{|H}}
  \geq 0.
\end{align}
By the same argument for (\ref{eq restc}), (\ref{eq induct0}) is equivalent to
\begin{equation}\label{eq induct}
  c\omega^{n-m-1}\wedge \alpha_1 \wedge...\wedge\alpha_{m-p-q}\wedge \Phi \wedge \overline{\Phi} \wedge i dH \wedge d\overline{H} \geq 0.
\end{equation}

Note that every $S(\alpha_j)$ is contained in a hyperplane, by Lemma \ref{lem basis} we can take an orthonormal basis $e_1, ...,e_n$ such that every
$$e_k \in \mathbb{C}^n \setminus \bigcup_{j=1}^{m-p-q} S(\alpha_j).$$
Applying (\ref{eq induct}) to $H_{e_j}$ and taking the sum over $j$ imply
\begin{equation}\label{eq induct1}
  c\omega^{n-m}\wedge \alpha_1 \wedge...\wedge\alpha_{m-p-q}\wedge \Phi \wedge \overline{\Phi} \geq 0,
\end{equation}
by using that
\begin{equation*}
  i \sum_j dH_{e_j} \wedge d\overline{H_{e_j}} = \omega.
\end{equation*}

By the assumption $\Omega \wedge \Phi =0$, (\ref{eq induct1}) is an equality, thus
$$Q_{H_{e_j}} (\Phi_{|H_{e_j}}, \Phi_{|H_{e_j}})=0$$
for every $j$. By induction, this yields that $\Phi_{|H_{e_j}} =0$ for every $j$.

We claim that this implies $\Phi =0$.

Without loss of generalities, we can assume $H_{e_j} (z) = z_j$. The claim can be proved by contradiction. Assume $$\Phi =\sum_{|I| =p, |J|=q} \Phi_{IJ} dz_I \wedge d\bar z_J \neq 0,$$
then there is a term $\Phi_{IJ} dz_I \wedge d\bar z_J \neq 0$. Since $p+q < m < n$, there must exist some $j$ such that the multi-indexes $I, J$ do not contain $j$. This implies $\Phi_{|H_{e_j}} \neq 0$, a contradiction.

This finishes the proof of the lemma.
\end{proof}

\begin{rmk}
Without the assumption on the semipositivity of $\alpha_j$, we associate $\alpha_j$ to the following open set
\begin{equation*}
  P(\alpha_j)=\{v\in \mathbb{C}^n | \alpha_j ^m \wedge \omega ^{n-m-1} \wedge idH_v \wedge d\overline{H_v} >0\}.
\end{equation*}
Then by Lemma \ref{lem restc}, it is clear that ${\alpha_j}_{|H_v}$ is $m$-positive with respect to $\omega_{|H_v}$ on $H_v$ for every $v\in P(\alpha_j)$. It is unclear to us whether the intersection $$\bigcap_{j=1} ^{m-p-q} P(\alpha_j) $$
always contains an othonormal basis. If this was true, then we could apply the induction as above and remove the semipositivity assumption. By Lemma \ref{lem basis}, this holds when the $\alpha_j$ are also semipositive.
\end{rmk}

\begin{lem}\label{nondeg}
In the same setting as Theorem \ref{thrm linear HRR}, assume that the HL holds for $\dim =n$, then the quadratic form $Q$ defined by $\Omega=\omega^{n-m}\wedge \alpha_1 \wedge ...\wedge \alpha_{m-p-q}$ is non-degenerate on $\Lambda^{p,q}$.
\end{lem}

\begin{proof}
This follows directly from Lemma \ref{lem HL}.
\end{proof}

Another consequence of the HL in dimension $n$ is the following LD.

\begin{lem}\label{LD}
In the same setting as Theorem \ref{thrm linear HRR}, assume that the HL holds for $\dim =n$, then the space $\Lambda^{p,q}$ has a $Q$-orthogonal direct sum decomposition
\begin{equation*}
 \Lambda^{p,q} = P^{p,q} \oplus (\alpha_{m-p-q+1} \wedge \Lambda^{p-1,q-1}),
\end{equation*}
where we use the convention that $\Lambda^{p-1,q-1} =\{0\}$ when $p=0$ or $q=0$.
Moreover, $$\dim P^{p,q} = \dim \Lambda^{p,q} - \dim \Lambda^{p-1,q-1},$$
which is independent of $\Omega$ and $\alpha_{m-p-q+1}$.
\end{lem}

\begin{proof}
By the assumption, the map $\Omega \wedge \alpha_{m-p-q+1} ^2 : \Lambda^{p-1,q-1} \rightarrow \Lambda^{n-q+1,n-p+1}$ is an isomorphism. Thus,
the map $\alpha_{m-p-q+1}: \Lambda^{p-1,q-1} \rightarrow \Lambda^{p,q}$ is injective, and
\begin{equation*}
P^{p,q} \cap (\alpha_{m-p-q+1} \wedge \Lambda^{p-1,q-1}) =\{0\}.
\end{equation*}
We also get that
$$\Omega \wedge \alpha_{m-p-q+1}: \Lambda^{p,q} \rightarrow \Lambda^{n-q+1,n-p+1}$$
is an isomorphism when restricted to $\alpha_{m-p-q+1} \wedge \Lambda^{p-1,q-1}$.
The kernel of $\Omega \wedge \alpha_{m-p-q+1}$ is exactly given by $P^{p,q}$. On the hand, note that
\begin{equation*}
  \dim (\alpha_{m-p-q+1} \wedge \Lambda^{p-1,q-1}) = \dim \Lambda^{p-1,q-1} =\dim \Lambda^{n-q+1,n-p+1}.
\end{equation*}
Thus, $\dim \Lambda^{p,q} = \dim P^{p,q} + \dim \alpha_{m-p-q+1} \wedge \Lambda^{p-1,q-1}$.

The orthogonality follows directly from the definition of $P^{p,q}$.
\end{proof}

\begin{proof}[Proof of Theorem \ref{thrm linear HRR}]
Assume that the HRR holds in $\dim = n-1$, then by Lemma \ref{lem HL} the HL holds in $\dim =n$. Thus we could apply Lemmas \ref{nondeg} and \ref{LD}. Theorem \ref{thrm linear HRR} will follow from a homotopy argument as in \cite{timorinMixedHRR}.
Consider the deformation
\begin{equation*}
  \Omega_t = \omega^{n-m} \wedge ((1-t)\alpha_1 + t \omega)\wedge...\wedge ((1-t)\alpha_{m-p-q} + t \omega),\ 0\leq t \leq 1.
\end{equation*}
The primitive space $P_t ^{p,q}$ is defined by $\Omega_t \wedge ((1-t)\alpha_{m-p-q+1} + t \omega)$. Note that every $(1-t)\alpha_1 + t \omega$ is $m$-positive and semipositive, thus all the quadratic forms $Q_t$ are non-degenerate and all the primitive subspaces $P_t ^{p,q}$ have the same dimension. When $t=1$, $Q_1$ is positive definite by the classical HRR. Thus $Q_0$ is also positive definite, since $Q_t$ never becomes degenerate in the course of deformation.

Therefore, the HRR holds in $\dim = n$.
This finishes the proof.
\end{proof}

The following local estimate is important when reducing the global case to the local case. 

Recall that the space $\Lambda^{p,q}$ admits an inner product defined by
\begin{equation*}
  \langle \Phi, \Psi \rangle = \sum_{I,J} \Phi_{IJ}\overline{\Psi_{IJ}},
\end{equation*}
where $\Phi = \sum_{I, J} \Phi_{IJ} dz_I \wedge d\bar z_J, \Psi = \sum_{I, J} \Psi_{IJ} dz_I \wedge d\bar z_J$. The norm is given by $||\Phi||^2 = \langle \Phi, \Phi \rangle$.

\begin{lem}\label{lem local estimate}
Using the same notations as Theorem \ref{thrm linear HRR}, there are positive constants $c_1, c_2$ such that
\begin{equation*}
  ||\Phi||^2 \leq c_1 Q(\Phi, \Phi) + c_2 ||\Omega\wedge \alpha_{m-p-q+1} \wedge \Phi||^2,\ \forall \Phi \in \Lambda^{p,q}.
\end{equation*}
\end{lem}

\begin{proof}
By Theorem \ref{thrm linear HRR} and Lemma \ref{LD}, the proof is the same as that of \cite[Proposition 2.2]{DN06} (see also \cite[Proposition 2.8]{DN13HRR}).
\end{proof}

\subsection{The global case}
Following \cite{DN06, DN13HRR}, the global case (Theorem \ref{thrm main result}) can be reduced to the local case (Theorem \ref{thrm linear HRR}) by solving a $dd^c$-equation.

We denote the space of smooth complex $(p,q)$ forms on $X$ by $\Lambda^{p,q}(X, \mathbb{C})$.

\begin{lem}\label{ddc eq}
In the same setting as Theorem \ref{thrm main result}, use the notations $\widehat{\omega}$ and $\widehat{\alpha}_1,...,\widehat{\alpha}_{m-p-q+1}$ to denote a K\"ahler metric and smooth $m$-positive and semipositive forms in the corresponding classes, then for any smooth form $\widehat{\Phi} \in \Lambda^{p, q}(X, \mathbb{C})$ such that its class $\{\widehat{\Phi}\}\in P^{p, q} (X, \mathbb{C})$, there is a smooth form $\widehat{F}$ such that
\begin{equation}\label{eq primitive}
  \widehat{\Omega} \wedge \widehat{\alpha}_{m-p-q+1}\wedge dd^c \widehat{F} =  \widehat{\Omega} \wedge \widehat{\alpha}_{m-p-q+1}\wedge \widehat{\Phi}.
\end{equation}
\end{lem}

\begin{proof}
Using Lemma \ref{lem local estimate}, this follows from \cite[Propositions 3.1, 3.2]{DN13HRR} (see also \cite{DN06}).
\end{proof}

Now we could prove Theorem \ref{thrm main result} by using Lemma \ref{ddc eq}.

\begin{proof}[Proof of Theorem \ref{thrm main result}]
Assume that $\Phi \in P^{p, q} (X, \mathbb{C})$ and let $\widehat{\Phi}$ be a smooth representative of the class $\Phi$. Then by Lemma \ref{ddc eq}, there is a smooth form $\widehat{F}$ such that
\begin{equation}\label{eq local primitive}
  \widehat{\Omega} \wedge \widehat{\alpha}_{m-p-q+1}\wedge (\widehat{\Phi} - dd^c \widehat{F}) =0.
\end{equation}
Thus $\widehat{\Phi} - dd^c \widehat{F}$ is a primitive $(p, q)$ form with respect to $\widehat{\Omega} \wedge \widehat{\alpha}_{m-p-q+1}$.

Applying Theorem \ref{thrm linear HRR}, we have
\begin{equation}\label{eq pt Q}
  c\widehat{\Omega} \wedge  (\widehat{\Phi} - dd^c \widehat{F}) \wedge \overline{(\widehat{\Phi} - dd^c \widehat{F})} \geq 0
\end{equation}
at every point.
By Stokes formula,
\begin{equation}\label{eq Q}
  Q(\Phi, \Phi) = c\int \widehat{\Omega} \wedge  (\widehat{\Phi} - dd^c \widehat{F}) \wedge \overline{(\widehat{\Phi} - dd^c \widehat{F})} \geq 0,
\end{equation}
where $c=i^{q-p} (-1)^{(p+q)(p+q+1)/2}$.

Moreover, if $Q(\Phi, \Phi)=0$, then we have equalities everywhere. In particular, (\ref{eq pt Q}) is an equality at every point. By Theorem \ref{thrm linear HRR} again, this yields
\begin{equation*}
  \widehat{\Phi} - dd^c \widehat{F} =0
\end{equation*}
on $X$. Thus $\Phi =0$ in $H^{p,q} (X, \mathbb{C})$. This finishes the proof of the global HRR.

As an immediate consequence, we get the global HL. We only need to check that $$\Omega: H^{p,q} (X, \mathbb{C})\rightarrow H^{n-q,n-p} (X, \mathbb{C})$$ is injective. Assume that $\Phi \in H^{p,q} (X, \mathbb{C})$ satisfyies $\Omega \cdot \Phi =0$, then $\Phi$ is primitive and $Q(\Phi, \Phi)=0$. By HRR, $\Phi =0$, which finishes the proof of the HL.

Finally, the global LD follows from similar arguments as Lemma \ref{LD}.

This finishes the proof of Theorem \ref{thrm main result}.

\end{proof}

\section{Further remarks}

\subsection{Abstract Hodge-Riemann forms}
In \cite{DN13HRR}, Dinh-Nguy\^{e}n gave an abstract version of HRR on compact K\"ahler manifolds by introducing the notion of Hodge-Riemann forms. By Theorem \ref{thrm main result}, a more general abstract mixed HRR can be established.

In the sequel, the $\alpha_j$ are assumed to have the same positivity as above, i.e., $m$-positivity and semipositivity.

\begin{defn}(analogous to \cite[Definition 2.1]{DN13HRR})
A real $(k,k)$ form $\Omega \in \Lambda ^{k,k}$, $k=n-p-q$, is called a Lefschetz form for the bidegree $(p,q)$ if the map
\begin{equation*}
  \Omega: \Lambda ^{p,q}\rightarrow  \Lambda ^{n-q,n-p},\ \Phi \mapsto \Omega \wedge \Phi
\end{equation*}
is an isomorphism.
Assume $p+q\leq m\leq n$, a real $(k,k)$ form is said to be a Hodge-Riemann form for the bidegree $(p,q)$ if there is a continuous deformation $\Omega_t \in \Lambda  ^{k,k} $ with $0\leq t\leq 1$, $\Omega_0 =\Omega$ and $\Omega_1 = \omega^{n-m}\wedge \alpha_1 \wedge...\wedge \alpha_{m-p-q}$ such that
\begin{equation*}
  \Omega_t \wedge \alpha_{m-p-q+1} ^{2r}\ \textrm{is a Lefschetz form for the bidegree}\ (p-r, q-r)
\end{equation*}
for every $r=0, 1$ and $0\leq t \leq 1$.
\end{defn}

\begin{defn}
Let $X$ be a compact K\"ahler manifold of dimension $n$. Then $\Omega\in H^{k,k} (X, \mathbb{R})$ is called a Hodge-Riemann class, if it has a representative which is a Hodge-Riemann form at every point.
\end{defn}

Analogous to \cite{DN13HRR}, we have:

\begin{thrm}
Let $X$ be a compact K\"ahler manifold of dimension $n$, and let $\Omega$ be a Hodge-Riemann class. Then the HRR holds with respect to $\Omega$, where the primitive space is given by $\Omega \cdot \alpha_{m-p-q+1}$.
\end{thrm}

\subsection{Generalized semi-small maps}\label{sec semismall}

Let $f: X\rightarrow Y$ be a proper surjective holomorphic map between two complex spaces. For every integer $k$ define
\begin{equation*}
  Y^k =\{y\in Y| \dim f^{-1}(y) = k\}.
\end{equation*}
The spaces $Y^k$ are analytic subvarieties of $Y$, whose disjoint union is $Y$.

Recall that $f$ is called a semi-small map in the sense of Goresky-MacPherson if $$\dim Y^k + 2k \leq \dim X$$ for every $k \leq \dim X /2$. Note that a semi-small map must be generically finite.

Semi-small maps can be generalized as follows.

\begin{defn}
We call a proper surjective holomorphic map $f: X\rightarrow Y$ \emph{relatively semi-small} if $$\dim Y^k + 2k \leq 2\dim X - \dim Y$$
for every $k$. Equivalently, $f$ is relatively semi-small if and only if there are no irreducible analytic subvarieties $T\subset X$ such that $$2 \dim T -2\dim X + \dim Y > \dim f(T).$$
\end{defn}

In particular, when $\dim X =\dim Y$, we get a semi-small map.

In \cite{cataldoMigHLsemismallmap}, a line bundle is called lef if the Kodaira map of its multiple induces a semi-small map. Corresponding to relatively semi-small maps, this can be generalized as follows.

\begin{defn}
A line bundle $L$ on a projective manifold $X$ is called \emph{relatively lef} if $kL$ is generated by its global sections for some positive integer $k$ and the corresponding morphism
\begin{equation*}
  \phi_{|kL|} : X \rightarrow Y = \phi_{|kL|} (X)
\end{equation*}
is a relatively semi-small map.
\end{defn}

Analogous to \cite{cataldoMigHLsemismallmap}, it is easy to get the following result.

\begin{prop}
Let $f: X\rightarrow Y$ be a surjective holomorphic map from a compact K\"ahler manifold of dimension $n$ to a projective variety $Y$ of dimension $m$. Let $\omega$ be a K\"ahler class on $X$, and let $A_1,...,A_{m-p-q}$ be ample line bundles on $Y$.
If the relative HL holds with respect to $$\omega^{n-m}\cdot f^* A_1 \cdot ...\cdot f^* A_{m-p-q}$$ for any $0\leq p,q\leq p+q \leq m$, then $f$ is a relatively semi-small map.
\end{prop}

\begin{proof}
Otherwise, assume that $f$ is not relatively semi-small. Then there is an irreducible analytic subvariety $T\subset X$ such that
\begin{equation*}
  2 \dim T -2n + m > \dim f(T).
\end{equation*}
Let $\{T\} \in H^{n-\dim T, n-\dim T} (X, \mathbb{R})$ be the fundamental class of $T$. The class $f^* A_1 \cdot ...\cdot f^* A_{m-2(n-\dim T)}$ can be represented by an analytic cycle that does not intersect $T$, thus its intersection with $\{T\}$ is zero. In particular,
\begin{equation*}
  \omega^{n-m}\cdot f^* A_1 \cdot ...\cdot f^* A_{m-2(n-\dim T)} \cdot \{T\} =0,
\end{equation*}
which implies that the relative HL does not hold for the bidegree $(n-\dim T, n-\dim T)$.

This finishes the proof.
\end{proof}

\begin{rmk}
In the beautiful paper \cite{cataldoMigHLsemismallmap}, de Cataldo-Migliorini proved that $L$ is lef on a projective manifold of dimension $n$ if and only if the HL holds with respect to $L^{n-p-q}$ for every $0\leq p,q\leq p+q\leq n$. Moreover, they proved that for semi-small maps, the deep Decomposition Theorem of Beilinson, Bernstein, Deligne and Gabber \cite{BBDG82DecompsitionThrm} is equivalent to the non-degeneracy of certain intersection forms (i.e., HRR) associated with a stratification. They applied this result to give a new proof of the Decomposition Theorem for the direct image of the constant sheaf.

In our setting, we expect that $ \phi_{|kL|} : X \rightarrow Y = \phi_{|kL|} (X)$ is relatively semi-small if and only if the relative HL holds with respect to $\omega^{n-m}\cdot L^{m-p-q}$ for every $0\leq p,q\leq p+q \leq m$. We intend to discuss it elsewhere.
Furthermore, we expect that the generalized form of Corollary \ref{cor relative HRR} might be applied to study the topology of these maps.
\end{rmk}

\bibliography{reference}

\providecommand{\bysame}{\leavevmode\hbox to3em{\hrulefill}\thinspace}
\providecommand{\MR}{\relax\ifhmode\unskip\space\fi MR }
\providecommand{\MRhref}[2]{%
  \href{http://www.ams.org/mathscinet-getitem?mr=#1}{#2}
}
\providecommand{\href}[2]{#2}
\begin{thebibliography}{McM93}

\bibitem[AHK18]{huhHRR}
Karim Adiprasito, June Huh, and Eric Katz, \emph{Hodge theory for combinatorial
  geometries}, Ann. of Math. (2) \textbf{188} (2018), no.~2, 381--452.
  \MR{3862944}

\bibitem[BBD82]{BBDG82DecompsitionThrm}
A.~A. Be\u{\i}linson, J.~Bernstein, and P.~Deligne, \emph{Faisceaux pervers},
  Analysis and topology on singular spaces, {I} ({L}uminy, 1981),
  Ast\'{e}risque, vol. 100, Soc. Math. France, Paris, 1982, pp.~5--171.
  \MR{751966}

\bibitem[BS02]{berdtsibony02dbarcurrent}
Bo~Berndtsson and Nessim Sibony, \emph{The {$\overline\partial$}-equation on a
  positive current}, Invent. Math. \textbf{147} (2002), no.~2, 371--428.
  \MR{1881924}

\bibitem[Cat08]{cattanimixedHRR}
Eduardo Cattani, \emph{Mixed {L}efschetz theorems and {H}odge-{R}iemann
  bilinear relations}, Int. Math. Res. Not. IMRN (2008), no.~10, Art. ID
  rnn025, 20. \MR{2429243}

\bibitem[dCM02]{cataldoMigHLsemismallmap}
Mark Andrea~A. de~Cataldo and Luca Migliorini, \emph{The hard {L}efschetz
  theorem and the topology of semismall maps}, Ann. Sci. \'{E}cole Norm. Sup.
  (4) \textbf{35} (2002), no.~5, 759--772. \MR{1951443}

\bibitem[Dem12]{Dem_AGbook}
Jean-Pierre Demailly, \emph{Complex analytic and differential geometry. online
  book}, available at www-fourier. ujf-grenoble. fr/~
  demailly/manuscripts/agbook. pdf, Institut Fourier, Grenoble (2012).

\bibitem[DN06]{DN06}
Tien-Cuong Dinh and Vi{\^e}t-Anh Nguy{\^e}n, \emph{The mixed {H}odge-{R}iemann
  bilinear relations for compact {K}\"ahler manifolds}, Geom. Funct. Anal.
  \textbf{16} (2006), no.~4, 838--849.

\bibitem[DN13]{DN13HRR}
Tien-Cuong Dinh and Vi\^{e}t-Anh Nguy\^{e}n, \emph{On the {L}efschetz and
  {H}odge-{R}iemann theorems}, Illinois J. Math. \textbf{57} (2013), no.~1,
  121--144. \MR{3224564}

\bibitem[DS04]{DinhS04aut}
Tien-Cuong Dinh and Nessim Sibony, \emph{Groupes commutatifs d'automorphismes
  d'une vari\'{e}t\'{e} k\"{a}hl\'{e}rienne compacte}, Duke Math. J.
  \textbf{123} (2004), no.~2, 311--328. \MR{2066940}

\bibitem[DS05]{dinhS05greencurrent}
\bysame, \emph{Green currents for holomorphic automorphisms of compact
  {K}\"{a}hler manifolds}, J. Amer. Math. Soc. \textbf{18} (2005), no.~2,
  291--312. \MR{2137979}

\bibitem[EW14]{williamsonHodgeSeorgel}
Ben Elias and Geordie Williamson, \emph{The {H}odge theory of {S}oergel
  bimodules}, Ann. of Math. (2) \textbf{180} (2014), no.~3, 1089--1136.
  \MR{3245013}

\bibitem[Gar59]{gardinghyperbolic}
Lars Garding, \emph{An inequality for hyperbolic polynomials}, J. Math. Mech.
  \textbf{8} (1959), 957--965. \MR{0113978}

\bibitem[Gro90]{gromov1990convex}
Misha Gromov, \emph{Convex sets and {K}\"ahler manifolds}, Advances in
  Differential Geometry and Topology, ed. F. Tricerri, World Scientific,
  Singapore (1990), 1--38.

\bibitem[Hor94]{hormanderConvexity}
Lars Hormander, \emph{Notions of convexity}, Progress in Mathematics, vol. 127,
  Birkh\"{a}user Boston, Inc., Boston, MA, 1994. \MR{1301332}

\bibitem[McM93]{mcmullenSimplePolytopes}
Peter McMullen, \emph{On simple polytopes}, Invent. Math. \textbf{113} (1993),
  no.~2, 419--444. \MR{1228132}

\bibitem[Tim98]{timorinMixedHRR}
V.~A. Timorin, \emph{Mixed {H}odge-{R}iemann bilinear relations in a linear
  context}, Funktsional. Anal. i Prilozhen. \textbf{32} (1998), no.~4, 63--68,
  96. \MR{1678857}

\bibitem[Tim99]{timorinPolytopeHRR}
\bysame, \emph{An analogue of the {H}odge-{R}iemann relations for simple convex
  polyhedra}, Uspekhi Mat. Nauk \textbf{54} (1999), no.~2(326), 113--162.
  \MR{1711255}

\bibitem[Voi07]{voisinHodge1}
Claire Voisin, \emph{Hodge theory and complex algebraic geometry. {I}}, english
  ed., Cambridge Studies in Advanced Mathematics, vol.~76, Cambridge University
  Press, Cambridge, 2007, Translated from the French by Leila Schneps.
  \MR{2451566}

\bibitem[Wil16]{williamson2016hodge}
Geordie Williamson, \emph{The {H}odge theory of the {H}ecke category}, arXiv
  preprint, arXiv:1610.06246 (2016).

\bibitem[Xia18]{xiao18Hodgeindex}
Jian Xiao, \emph{Hodge-index type inequalities, hyperbolic polynomials and
  complex {H}essian equations}, arXiv preprint, arXiv:1810.04662 (2018).

\end{thebibliography}
\bibliographystyle{amsalpha}

\bigskip

\bigskip

\noindent
\textsc{Tsinghua University, Beijing, China}\\
\noindent
\verb"Email: jianxiao@mail.tsinghua.edu.cn"

\end{document}